\documentclass[12pt]{amsart}
\subjclass{55Q52}
\usepackage{amsmath,amssymb}

\usepackage[h margin=1 in, v margin=1 in]{geometry}
\usepackage{amsthm}
\usepackage{enumitem}
\usepackage{tikz-cd} 
\usepackage{comment}

\newtheorem{thm}{Theorem}[section]
\newtheorem*{thm*}{Theorem}
\newtheorem{cor}[thm]{Corollary}

\newtheorem{lem}[thm]{Lemma}
\newtheorem*{lem*}{Lemma}

\theoremstyle{definition}

\newtheorem*{defn*}{Definition}
\newtheorem*{defi}{Definition}

\theoremstyle{remark}

\newtheorem*{rem*}{Remark}

\numberwithin{equation}{section}
\graphicspath{{images/}}
\newcommand{\inv}[1]{#1^{-1}}
\newcommand{\bb}[1]{\mathbb{#1}}
\newcommand{\T}{\intercal}
\newcommand{\R}{\mathbb R}

\newcommand{\Ima}{\text{Im }}

\title[Homotopy groups of quasi-spheres and applications]{Homotopy groups of quasi-spheres and applications to indefinite orthogonal groups}
\author{Xiangjia Kong, Reese Lance, and Franklin Rea}

\begin{document}
\maketitle
\begin{abstract}In this note, we present a new proof of the isomorphism $\pi_1(SO^+(p,q)) \cong \pi_1(SO(p))\times \pi_1(SO(q))$ using the long exact sequence associated to a fibration. While this formula is already known, the method of proof presented here fills an existing hole in the literature to naturally generalize an approach following the classical reduction of this formula for $q=0$, involving the homotopy groups of the $(p-1)$-sphere and a long exact sequence arising from a fibration constructed from these spaces. To generalize for $q \ne 0$, we upgrade spheres to the corresponding quasi-spheres, then analyze the resulting long exact sequence to obtain the isomorphism for all $p,q$. Some low-dimensional inductive steps are delicate in this approach, but it is more direct than other standard methods and involves some interesting calculations with fundamental groups. The question of whether or not this approach would work is absolutely natural, and in this paper we show that the answer is yes. This particular approach avoids some difficult technical steps in other standard proofs, which we review at the end. The family of spaces we consider are also of interest in physical applications. We remark about this direction to provide some context, but do not explore it in depth in this paper.
\end{abstract}

\section{Introduction}

As motivation, we first outline one standard calculation of $\pi_1(SO(n))$. Consider the fibration 
\[
SO(n-1) \to SO(n)\to S^{n-1}
\]
with inclusion and action by restricting the fundamental representation. This induces a long exact sequence on homotopy groups
\[
\cdots \to \pi_{2}(S^{n-1}) \to \pi_1(SO(n-1)) \to \pi_1(SO(n)) \to \pi_1(S^{n-1}) \to \cdots.
\]
Computing the higher homotopy groups of $SO(n)$ with this approach would involve computing terms such as $\pi_k(S^n)$, which is generally not known. However, for $\pi_1(SO(n))$ we only need to look at the portion of the LES depicted above, for which all the values are known: $\pi_k(S^n) = 0$ for $k < n$ and $\pi_k(S^n) = \bb Z$ for $k = n$. The only remaining sector is $k>n$, which remains open. We conclude that for $n> 3$, the terms on the left and right of the LES are $0$, so that $\pi_1(SO(n)) \cong \pi_1(SO(n-1))$, thus the fundamental group stabilizes at $n=3$. There are standard isomorphisms of spaces inducing $\pi_1(SO(3)) \cong \pi_1(\bb R \bb P^3) \cong \bb Z/2\bb Z$ and $\pi_1(SO(2)) \cong \pi_1(S^1) \cong \bb Z$. 
So we have
\[
\pi_1(SO(n)) = \begin{cases}
0& n = 1 \\ 
\bb Z& n = 2 \\
\bb Z/2 \bb Z & n\geq 3.
\end{cases}
\]
Inspired by this strategy, we wish to do something similar to compute the fundamental group of $SO^+(p,q)$. In Section 3, we introduce the family of spaces known as quasi-spheres, then compute their fundamental groups using a LES from a fibration, and provide a different approach to show this.  These quasi-spheres are of interest independent of their use for our applications here, particularly for physicists. In Section 4, we construct an analogous fibration to the one given for $SO(n)$, replacing (definite) special orthogonal groups with indefinite special orthogonal groups and promoting spheres to quasi-spheres. We remark in this section how the spaces investigated in this paper arise when doing physics, but we do not investigate this angle in depth in this paper. From the fibration, we employ the resulting long exact sequences to prove $\pi_1(SO^+(p,q)) \cong \pi_1(SO(p)) \times \pi_1(SO(q))$. For comparison, we provide two standard calculations of $\pi_1(SO^+(p,q))$ in Section 5. These involve non-trivial steps: The calculation of maximal compact subgroups, $K$, for $G = SO^+(p,q)$, the calculation of the corresponding symmetric spaces $G/K$, and appeals to general decomposition theorems, whose proofs we do not provide. One advantage of the approach taken in this note is to avoid these intensive steps. It is of interest to see how the usage of exact sequences, via some special considerations, result in the ultimate (classical) formula. Having seen the calculation of $\pi_1(SO(n))$ provided above, it is absolutely natural to ask if the approach we describe in this paper would work for $SO^+(p,q)$, and here we show that the answer is yes. Likely, attempts in this direction were abandoned due to ramification of some low-dimensional steps, which we will show. It is interesting to see that in spite of these exceptional cases being quite complicated, in the end one obtains a very simple and elegant formula. 
\section{Acknowledgements}
The first author was supported in part by the NSF, grant DMS-1555206. All three authors thank Ivan Cherednik for posing this question and his help with this note. The support of his NSF grant DMS-1901796 is acknowledged. 
\section{Homotopy groups of quasi-spheres}
\begin{defi}The $n$-dimensional \textit{quasi-sphere} is defined by
\[
    X_{p,q}^{\pm} = \{x \in \bb R^{n+1}\ |\ x_1^2 + \dots + x_p^2 - x_{p+1}^2 - \dots - x_{p+q}^2 = \pm 1\}.
\]
\end{defi}
It is also realized as the unit sphere in $\bb R^{p,q}$, whose underlying linear space is $\bb R^{p+q+1}$, equipped with the standard quadratic form of signature $(p,q)$ in the case of $+$ and $(q,p)$ in the case of $-$. Quasi-spheres themselves have been known to mathematicians and physicists for quite some time, such that there are now standard textbooks, see \cite{porteous}, \cite{vaz}. The definition we provide here is a particular instance of a more general definition, also taken in \cite{porteous}, \cite{vaz}, allowing for more general quadratic forms, but we do not investigate these generalizations here. 

\begin{rem*}Quasi-spheres are of physical interest, separate from the applications we explore here. The family of de Sitter and anti-de Sitter spaces, which fit into the very deep story of AdS/CFT correspondence and gravity at large, are realized as quasi-spheres: $dS$ corresponds to $X^-$ while AdS corresponds to $X^+$. Some authors define dS/AdS spaces to be the universal covers of the corresponding quasi-sphere, making their fundamental groups also of interest. For example, if any of these dS or AdS spaces happen to be simply connected, then they are their own universal cover.
\end{rem*}

We want to calculate some low dimensional homotopy groups of these quasi-spheres for use in the next section, using the long exact sequence associated to a fibration. At the end of this section, we provide an alternative, faster route to compute these groups, namely exhibiting a homeomorphism $X_{p,q}^+ \cong \bb R^q \times S^{p-1}$.\\

There is a natural projection
\begin{align*}
\pi: X_{p,q}^+ &\to \bb R^q\\
(x_1,\dots,x_p,x_{p+1},\dots,x_{p+q}) &\mapsto (x_{p+1},\dots,x_{p+q})
\end{align*}
out of which we want to build a Serre fibration. For an introductory treatment on the theory of fibrations, including how to construct a long exact sequence of homotopy groups associated to any fibration, see \cite{hatcher} chapter 4, section 2. We will use freely the basic definitions and results from this source. \\\\
The fibers of the above projection for a given $(z_1,\dots,z_n) \in \bb R^q$ are
\[
    \inv\pi(z_1,\dots,z_n) = \{x\in X_{p,q}^+\ |\ x_1^2+\dots+x_p^2 = 1+z_1^2 + \dots + z_{p+q}^2\},
\]
which is isomorphic to the sphere of radius $\sqrt{1+z_1^2 + \dots + z_{p+q}^2}$. This gives the fibration
\[
    S^{p-1} \hookrightarrow X_{p,q}^+ \to \bb R^q,
\]
\begin{rem*}One would show that what we wrote down is a fiber bundle, and it is a valuable exercise to work out local triviality conditions. It is known that every fiber bundle yields a Serre fibration, for example it is proved in \cite{hatcher}.
\end{rem*}
In our case, the long exact sequence provided is 
\[
\begin{tikzcd}
& \cdots            \arrow[r]           &                            \pi_3(\bb R^q) \ar[overlay, out=-10, in=170]{dll} &    \\
\pi_2(S^{p-1}) \arrow[r] & {\pi_2(X_{p,q}^+)} \arrow[r] & \pi_2(\bb R^q) \ar[overlay, out=-10, in=170]{dll}  &   \\
\pi_1(S^{p-1}) \arrow[r] & {\pi_1(X_{p,q}^+)} \arrow[r] & \pi_1(\bb R^q) \ar[overlay, out=-10, in=170]{dll}  &   \\
\pi_0(S^{p-1}) \arrow[r] & {\pi_0(X_{p,q}^+)} \arrow[r] & \pi_0(\bb R^q) \arrow[r]   & 0.
\end{tikzcd}
\]
We can simplify the above sequence using the following facts, whose proofs are also found in \cite{hatcher}:
\begin{itemize}
    \item $\pi_k(\bb R^n) \cong \pi_k(\{*\}) = 0$ for all $k \neq 0$.
    \item $\pi_1(S^n) = 0$ for all $n \geq 2$, $\pi_2(S^n) = 0$ for all $n \neq 2$ and $\pi_2(S^2) = \bb Z$.
\end{itemize}
Substituting in yields
\[
\begin{tikzcd}
& \cdots            \arrow[r]           &                            0 \ar[overlay, out=-10, in=170]{dll} &    \\
\pi_2(S^{p-1}) \arrow[r] & {\pi_2(X_{p,q}^+)} \arrow[r] & 0 \ar[overlay, out=-10, in=170]{dll}  &   \\
\pi_1(S^{p-1}) \arrow[r] & {\pi_1(X_{p,q}^+)} \arrow[r] & 0 \ar[overlay, out=-10, in=170]{dll}  &   \\
\pi_0(S^{p-1}) \arrow[r] & {\pi_0(X_{p,q}^+)} \arrow[r] & 0 \arrow[r]   & 0.
\end{tikzcd}
\]
In the zeroth degree, we have the short exact sequence
\[
0 \to \pi_0(S^{p-1}) \xrightarrow{f} \pi_0(X_{p,q}^+) \to 0 
\]
which implies the two non-zero sets\footnote{In general $\pi_0$ does not have a group structure.} are in bijection. $\pi_0(S^{p-1})$ is a set with cardinality equal to the number of path components of $S^{p-1}$, which is path-connected for $p \neq 1$ and has two path components for $p = 1$:
\[
    \pi_0(X_{p,q}^+) = \pi_0(S^{p-1}) = \begin{cases}
    \bb Z / 2 \bb Z & p = 1\\
    0 & p \ne 1.
        \end{cases}
\]
To compute the higher degree homotopy groups, we have the following cases for $p$:
\begin{itemize}
    \item[($p\geq4$)] We have $\pi_k(S^{p-1}) = 0$ for $k=0,1,2$, so the long exact sequence in the $k$-th degree will give us
\[
    0 \xrightarrow{f} \pi_k(X_{p,q}^+) \xrightarrow{g} 0.
\]
Exactness at the middle term implies $\pi_k(X_{p,q}^+) = \ker g = \Ima f = 0$, for $k = 0,1,2$ and $p\geq4$.\\

    \item[($p=3$)] We have 
\[
\begin{tikzcd}
                 & \cdots                       & 0 \ar[overlay, out=-10, in=170]{dll} \\
\bb Z  \arrow[r] & {\pi_2(X_{3,q}^+)} \arrow[r] & 0 \ar[overlay, out=-10, in=170]{dll} \\
0 \arrow[r]      & {\pi_1(X_{3,q}^+)} \arrow[r] & 0            
\end{tikzcd}
\]
from which we obtain the short exact sequence
\[
0 \to \bb Z \to \pi_2(X_{3,q}^+) \to 0
\]
which implies $\pi_2(X_{3,q}^+) \cong \bb Z$, following the same argument we made for $\pi_0(X)$. \\
We also obtain from above
\[
    0 \to \pi_1(X_{3,q}^+) \to 0
\]
which implies $\pi_1(X_{3,q}^+) =0$. So we conclude
\[
    \pi_k(X_{3,q}^+) = \begin{cases}
    \bb Z & k=2\\
    0 & k = 1 \text{ or } 0.
    \end{cases}
\]
    \item[($p=2$)] We have 
\[
\begin{tikzcd}
                 & \cdots                       & 0 \ar[overlay, out=-10, in=170]{dll} \\
0  \arrow[r] & {\pi_2(X_{2,q}^+)} \arrow[r] & 0 \ar[overlay, out=-10, in=170]{dll} \\
\bb Z \arrow[r]      & {\pi_1(X_{2,q}^+)} \arrow[r] & 0,       
\end{tikzcd}
\]
using the same argument as the case $p = 3$, we get the analogous result
\[
    \pi_k(X_{2,q}^+) = \begin{cases}
    0 & k = 2\\
    \bb Z & k = 1 \\
    0 & k = 0.
    \end{cases}
\]
    \item[($p=1$)] We have 
\[
\begin{tikzcd}
                 & \cdots                       & 0 \ar[overlay, out=-10, in=170]{dll} \\
0  \arrow[r] & {\pi_2(X_{1,q}^+)} \arrow[r] & 0 \ar[overlay, out=-10, in=170]{dll} \\
\bb Z / 2 \bb Z \arrow[r]      & {\pi_1(X_{1,q}^+)} \arrow[r] & 0       
\end{tikzcd}
\]
which gives
\[
    \pi_k(X_{1,q}^+) = \begin{cases}
    0 & k = 1 \text{ or } 2\\
    \bb Z / 2 \bb Z & k = 0.
    \end{cases}
\]
\end{itemize}
Putting everything together, 
\[
\pi_0(X^+_{p,q}) = \begin{cases}
0 & \text{if } p \neq 1\\
\mathbb{Z}/ 2 \bb Z & \text{if } p = 1\\
\end{cases}
, \quad \pi_1(X^+_{p,q}) = \begin{cases}
0 & \text{if } p \neq 2\\
\mathbb{Z} & \text{if } p = 2\\
\end{cases}
, \quad \pi_2(X^+_{p,q}) = \begin{cases}
0 & \text{if } p \neq 3\\
\mathbb{Z} & \text{if } p = 3.
\end{cases}
\]
One may notice that these resemble the homotopy groups of $S^{p-1}$. Indeed we could also calculate these homotopy groups by defining the map suggested in \cite{porteous} (modulo variation in notation)
\begin{align*}
\bb R^q \times S^{p-1} &\to X^+_{p,q}\\
(x,y) &\mapsto \Big((\sqrt{1+x\cdot x})y,x\Big)
\end{align*}
and it is elementary to show that this map is bijective, continuous and with continuous inverse. From this it follows that $\pi_n(X_{p,q}^+) \cong \pi_n(S^{p-1})$ via a deformation retract. 
\section{The fundamental group of $SO^+(p,q)$ via long exact sequences and fibrations}
The introductory definitions and results from this section are covered in, for example, \cite{porteous}, but the notation used there is quite different from ours, so we repeat some of the relevant items here. \\\\
\begin{defi}The \textit{indefinite special orthogonal group}, $SO(p,q)$, is the subgroup of $GL_n(\bb R )$ which preserves the symmetric bilinear form given by the matrix
\[
I_{p,q}= \text{diag}(\underbrace{1,\dots,1}_p, \underbrace{-1,\dots,-1}_{q}).
\]
Explicitly it is 
\[
SO(p,q) = \big\{A \in GL_n(\bb R ) : \inv A = I_{p,q} A^\intercal I_{p,q}, \; \det A =  1\big\},
\]
and the analogy to the classical group $SO(n)$ is clear.
\end{defi}
We want to establish the fibration 
\begin{equation}
    SO(p-1,q) \hookrightarrow SO(p,q) \to X_{p,q}^+
\end{equation}
in order to calculate the fundamental groups via the resulting long exact sequence. There is an inclusion map given by
\begin{align*}
    i : SO(p-1,q) &\hookrightarrow SO(p,q)\\
    A &\mapsto \begin{pmatrix} 1 & 0\\ 0& A \end{pmatrix},
\end{align*}
the image is contained in $SO(p,q)$ since
\[
    i(A) I_{p,q} i(A)^\intercal = \begin{pmatrix}1 & 0\\0 & A \end{pmatrix} \begin{pmatrix} 1 & 0\\0 & I_{p-1,q} \end{pmatrix}  \begin{pmatrix}1 & 0\\0 & A^\intercal \end{pmatrix} = \begin{pmatrix} 1 & 0\\0 & I_{p-1,q} \end{pmatrix} = I_{p,q}.
\]
There is also a map
\begin{align*}
    \pi: SO(p,q) &\to X_{p,q}^+\\
    A &\mapsto A e_1 \equiv A_1
\end{align*}
where $\{e_i\}$ is the standard basis of $\R^n$ and $A_1$ denotes the first column of $A$. To show the image is contained in $X_{p,q}^+$, we first decompose $A$ into blocks
\[
A = \begin{pmatrix} B & C \\ D & E \end{pmatrix},
\]
where $B,C,D,E$ are $p\times p, q\times p, p\times q, q\times q$ matrices respectively. Then by the definition of $SO(p,q)$, we have
\begin{equation}
A^{-1} = I_{p,q} A^\T  I_{p,q} = \begin{pmatrix}
B^\intercal & -D^\intercal\\
-C^\intercal & E^\intercal
\end{pmatrix}. 
\end{equation}
Write the first column of $A$ as 
\[
A_1 = \begin{pmatrix}
b_1 & \cdots & b_p & d_1 & \cdots & d_q
\end{pmatrix}^\T
\]
$(4.2)$ then gives the first row of $A^{-1}$ as 
\[
(A^{-1})^1 = \begin{pmatrix}
b_1&\dots&b_p&-d_1&\dots& -d_q
\end{pmatrix}
\]
so the entry in the first row and first column of $A^{-1}A = I$ is
\[
1 = (I)_{11} = (A^{-1}A)_{11} = (A^{-1})^1 A_1 = b_1^2 + \dots + b_p^2 - d_1^2 - \dots - d_q^2
\]
which implies that $\pi(A)=A_1 \in X^+_{p,q}$ as a vector in $\mathbb{R}^{p+q}$.

Now we calculate the fiber $\inv \pi (x)$ for some $x \in X_{p,q}^+$. By definition,
\[
    \inv \pi (x) = \big\{A \in SO(p,q) : A e_1 = x \big\}
\]
First $\inv\pi (e_1)$ consists of matrices with $e_1$ as its first column, so if $A \in \inv\pi (e_1)$, i.e. $A e_1 = e_1$, then we have $\inv A e_1 = e_1$, which means
\[
e_1 = I_{p,q} e_1 = I_{p,q} \inv A e_1 = A^\T  I_{p,q} e_1 = A^\T e_1 = (A^\T)_1.
\]
Thus elements of $\inv \pi(e_1)$ are exactly 
\[
A = \begin{pmatrix} 1 & 0 \\ 
0 & A' \end{pmatrix} = i(A') \in SO(p,q)
\]
for some $A' \in SO(p-1,q)$, so that $\inv\pi(e_1) \cong SO(p-1,q)$.  \\\\
\begin{lem}\textit{For any $x \in X_{p,q}^+$, there exists a matrix $A_x \in SO(p,q)$ such that $A_x(e_1) = x$ (and thus the action of $SO(p,q)$ on $X_{p,q}^+$ is transitive).}
\end{lem}
To prove this, we first need a lemma characterizing the columns of a matrix in $O(p,q)$, generalizing the result that the columns of an orthogonal matrix form an orthonormal set with respect to the dot product:\\
\begin{lem}\textit{$A \in O(p,q)\iff $ the columns of $A$ form an orthogonal set with respect to the bilinear form of signature $(p,q)$ and there are $p$ columns of $(p,q)$ magnitude $1$ and $q$ columns of $(p,q)$ magnitude $-1$. }
\end{lem}
\begin{proof}[Proof of Lemma 4.2] 
Define $D = \sqrt{I_{p,q}} = \text{diag}(1,\dots,1,i,\dots,i)$ and observe that $DAD$ is a complex orthogonal (NOT unitary) matrix. Therefore the columns of $DAD$ form an orthonormal set with respect to the \textit{dot product} (NOT the standard inner product) on $\bb C^{p+q}$. But this is exactly equivalent to the desired condition. 
\end{proof}

\begin{proof}[Proof of Lemma 4.1] 
Let the first column of $A_x$ be the coefficients of $x$ in the standard basis. Because the bilinear form is non-degenerate, we have $\bb R^{p+q} = \langle x \rangle + \langle x \rangle ^\perp$. In fact it is a direct sum because the intersection consists of magnitude 0 vectors, but must also be contained in $\langle x \rangle$, which contains no magnitude 0 vectors. Assume $\langle x \rangle^\perp\ne 0$, otherwise $p$ or $q$ is 0 and the other is 1, and the lemma would be proven. Denote the standard symmetric bilinear form of signature $(p,q)$ as $Q$. It must be that $\text{dim }\langle x\rangle^\perp = p+q-1$ and the signature of $Q|_{\langle x\rangle^\perp}$ is $(p-1,q)$, so choose $x_1 = x$ and $x_2$ some magnitude 1 vector in $\langle x \rangle^\perp$, and place these in the first two columns. To fill in the third column, we observe that $\bb R^{p,q} = \langle x_1,x_2\rangle \oplus \langle x_1,x_2\rangle ^\perp$: The intersection is given by 
\begin{gather*}
\Big\{ ax_1 + bx_2\ |\ \forall u,w,\ \langle ax_1+bx_2,ux_1+wx_2\rangle = 0\Big\}\\
= \Big\{ ax_1 + bx_2\ |\ \forall u,w,\ au + bw = 0\Big\}\\
= \{0\}
\end{gather*}
$Q|_{\langle x_1,x_2\rangle^\perp} = (p-2,q)$, so we may choose a magnitude one vector in $\langle x_1,x_2\rangle^\perp$ and place it in column 3 of $A_x$. An inductive argument fills out the rest of the matrix with $p$ columns of magnitude 1 and $q$ columns of magnitude $-1$, all of which are pairwise orthogonal with respect to the $(p,q)$ bilinear form. The matrix is orthogonal so it has determinant $\pm 1$ (it is easy to show this even for orthogonal groups of indefinite forms), and if it is $-1$ we may swap two adjacent columns besides the first to obtain the desired matrix in $SO(p,q)$.
\end{proof}

\begin{cor}$A_x$ defines a homeomorphism $\inv\pi(e_1) \cong \inv\pi(x)$, so all fibers are homeomorphic to $SO(p-1,q)$ and
\[
SO(p-1,q) \to SO(p,q) \to X_{p,q}^+
\]
is a fibration.
\end{cor}
There is also a second fibration 
\begin{equation}
SO(p,q-1) \hookrightarrow SO(p,q) \rightarrow X^-_{p,q} = X^+_{q,p},
\end{equation}
which can be shown by the same argument but embedding $A \in SO(p,q-1)$ via
\[
\begin{pmatrix}
A & 0\\
0 & -1
\end{pmatrix} \in SO(p,q)
\]
and looking at the last column. Also $X^-_{p,q} = X^+_{q,p}$ is given by the fact that
\[
x_1^2 + \dots + x_p^2 - x_{p+1}^2 - \dots - x_{p+q}^2 = -1 \Longleftrightarrow x_{p+1}^2 + \dots + x_{p+q}^2 - x_1^2 - \dots - x_p^2  = 1.
\]
Using these two fibrations, we can now prove the main result. To reiterate, we do not claim any originality of this result, but believe the proof itself is of interest. 

\begin{thm}For all $p, q\ge 0$,
\[\pi_1(SO^+(p,q)) \cong \pi_1(SO(p)) \times \pi_1(SO(q)). \]
\end{thm}

\begin{rem*}Since $SO(p,q)$ has two connected components \cite{ov}, p. 43, we restrict to $SO^+(p,q)$, the identity component of $SO(p,q)$. In the physics literature, $SO(p,q)$ is known as the \textit{Lorentz group} (sometimes only with $p=1,q=3)$. This group is of interest to physicists because it is realized as a subgroup of isometries of Minkowski space, namely those which fix the origin. These transformations fix null vectors, geodesics of light in spacetime. The identity component, those transformations which preserve the orientation of both time and space, is known as the \textit{orthochronous Lorentz group}, or \textit{restricted Lorentz group}. When adding in (via a semidirect product) translations, one obtains the \textit{Poincare group}, which is the full symmetry group of special relativity. Obviously null vectors themselves are no longer fixed, but the set of all such is preserved. Higher values of $p$ and $q$ may be of interest to physicists working in high dimension, such as string theories.
\end{rem*}

\begin{proof} We now assume without loss of generality $p \geq q$, since we have $SO^+(p,q) \cong SO^+(q,p)$. The fibration (3.1) induces the fibration on connected components
\[
SO^+(p-1,q) \hookrightarrow SO^+(p,q) \rightarrow X^+_{p,q}, \quad \text{ for } p \geq 2,
\]
because $\text{Id}_{p-1+q} \mapsto \text{Id}_{p+q}$ and $SO^+(p-1,q)$ is connected, the inclusion map is continuous, thus its image is a connected subset which contains the identity $\text{Id}_{p+q}$. This gives a long exact sequence on homotopy groups
\[
\begin{tikzcd}
& \cdots \arrow[r] &  \pi_2(X^+_{p,q}) \ar[overlay, out=-10, in=170]{dll} &    \\
\pi_1(SO^+(p-1,q)) \arrow[r] & {\pi_1(SO^+(p,q))} \arrow[r] & \pi_1(X^+_{p,q}) \ar[overlay, out=-10, in=170]{dll}  &   \\
\pi_0(SO^+(p-1,q)) \arrow[r] & {\pi_0(SO^+(p,q))} \arrow[r] & \pi_0(X^+_{p,q}) \arrow[r]   & 0.
\end{tikzcd}
\]
We use this long exact sequence, along with the fundamental groups $\pi_k(X^+_{p,q})$ computed in section 2 to prove our result. We proceed by induction when $p > 3$, and prove the ``base cases'' of $p \leq 3$ directly.
\begin{itemize}
    \item [($p = 1$)] Since $\pi_1(X^+_{1,q}) = \pi_2(X^+_{1,q}) = 0$, we get the exact sequence
\[
0 \rightarrow \pi_1(SO^+(0,q)) = \pi_1(SO(q)) \rightarrow \pi_1(SO^+(1,q)) \rightarrow 0.
\]
From the exactness of the sequence above we get
\begin{align*}
&\boxed{\pi_1(SO^+(1,q))} = \pi_1(SO^+(0,q)) = \pi_1(SO(q)) \\
&\qquad\qquad\qquad\qquad\qquad = 0 \times \pi_1(SO(q)) = \boxed{\pi_1(SO(1)) \times \pi_1(SO(q))}
\end{align*}
as required. Note that since $SO(1)$ consist of only one element we must have $\pi_1(SO(1)) = 0$.\\

\item [($p = 2$)] Since $\pi_1(X^+_{2,q}) = \mathbb{Z}, \pi_2(X^+_{2,q}) = 0$, we get the short exact sequence
\[
0 \rightarrow \underbrace{\pi_1(SO^+(1,q)) = \pi_1(SO(q))}_\textrm{from $p=1$ case above} \hookrightarrow \pi_1(SO^+(2,q)) \rightarrow \mathbb{Z} \rightarrow 0
\]
which gives
\[
\boxed{\pi_1(SO^+(2,q))} = \mathbb{Z} \times \pi_1(SO^+(1,q)) = \boxed{\pi_1(SO(2)) \times \pi_1(SO(q))}
\]
since $SO(2) \cong S^1$ implies $\pi_1(SO(2)) = \mathbb{Z}$.\\

\item [($p = 3$)] Similar to above, restricting to $SO^+(p,q)$ on fibration (3.2) gives
\[
SO^+(3,q-1) \hookrightarrow SO^+(3,q) \rightarrow X^-_{3,q} = X^+_{q,3}
\]
which gives the long exact sequence
\[
\begin{tikzcd}
& \cdots \arrow[r] &  \pi_2(X^+_{q,3}) \ar[overlay, out=-10, in=170]{dll} &    \\
\pi_1(SO^+(3,q-1)) \arrow[r] & {\pi_1(SO^+(3,q))} \arrow[r] & \pi_1(X^+_{q,3}) \ar[overlay, out=-10, in=170]{dll}  &   \\
\pi_0(SO^+(3,q-1)) \arrow[r] & {\pi_0(SO^+(3,q))} \arrow[r] & \pi_0(X^+_{q,3}) \arrow[r]   & 0.
\end{tikzcd}
\]
We will use this to prove our result by considering cases on $q$. We know 
\[
\pi_1(SO^+(3,0)) = \pi_1(SO(3)) = \mathbb{Z}/2\mathbb{Z}
\]
so in the case of $q = 1$ we get
\[
\begin{tikzcd}
& \cdots \arrow[r] &  \pi_2(X^+_{1,3}) = 0 \ar[overlay, out=-10, in=170]{dll} \\
\pi_1(SO^+(3,0)) \arrow[r] & {\pi_1(SO^+(3,1))} \arrow[r] & \pi_1(X^+_{1,3})=0 
\end{tikzcd}
\]
which simplifies to
\[
0 \rightarrow \pi_1(SO^+(3,0)) \rightarrow \pi_1(SO^+(3,1)) \rightarrow 0
\]
to gives us
\[
\boxed{\pi_1(SO^+(3,1))} = \pi_1(SO^+(3,0)) = \pi_1(SO(3)) = \boxed{\pi_1(SO(3)) \times \pi_1(SO(1))}.
\]
 We do the same for $q =2$ to get
\[
\begin{tikzcd}
& \cdots \arrow[r] &  \pi_2(X^+_{2,3}) = 0 \ar[overlay, out=-10, in=170]{dll} \\
\pi_1(SO^+(3,1)) \arrow[r] & {\pi_1(SO^+(3,2))} \arrow[r] & \pi_1(X^+_{2,3})= \bb Z \ar[overlay, out=-10, in=170]{dll}\\
\underbrace{\pi_0(SO^+(3,1)) = 0}_{\text{path-connected}} &&
\end{tikzcd}
\]
which gives us
\[
0 \rightarrow \pi_1(SO^+(3,1)) \rightarrow \pi_1(SO^+(3,2)) \rightarrow \mathbb{Z} \rightarrow 0
\]
to get
\[
\boxed{\pi_1(SO^+(3,2))} = \pi_1(SO^+(3,1)) \times \mathbb{Z} = \boxed{\pi_1(SO(3)) \times \pi_1(SO(2))}
\]
where $\pi_1(SO^+(3,1)) = \pi_1(SO(3))$ from above. The case of $q = 3$, i.e.  $\pi_1(SO^+(3,3))$, is more involved than the rest, so it is dealt with separately later. For now we will assume it. If $q > 3$ then we have
\[
\begin{tikzcd}
& \cdots \arrow[r] &  \pi_2(X^+_{q,3}) = 0 \ar[overlay, out=-10, in=170]{dll} \\
\pi_1(SO^+(3,q-1)) \arrow[r] & {\pi_1(SO^+(3,q))} \arrow[r] & \pi_1(X^+_{q,3}) = 0 \ar[overlay, out=-10, in=170]{dll}\\
\pi_0(SO^+(3,q-1)) = 0 &&
\end{tikzcd}
\]
which simplifies to
\[
0 \rightarrow \pi_1(SO^+(3,q-1)) \rightarrow \pi_1(SO^+(3,q)) \rightarrow 0
\]
to get 
\[
\pi_1(SO^+(3,q)) = \pi_1(SO^+(3,q-1)) = \dots = \pi_1(SO^+(3,3))
\]
for $q > 3$, and since $\pi_1(SO(q-1)) = \pi_1(SO(q))$ for $q>3$ we get the desired
\begin{align*}
\boxed{\pi_1(SO^+(3,q))} &= \pi_1(SO^+(3,3))\\
&= \pi_1(SO(3)) \times \pi_1(SO(3)) = \boxed{\pi_1(SO(3)) \times \pi_1(SO(q))}.
\end{align*}

\item[($p > 3$)] We induct on $p$. Since $\pi_1(X^+_{p,q}) = \pi_2(X^+_{p,q}) = 0$ for $p > 3$, we get the exact sequence
\[
0 \rightarrow \pi_1(SO^+(p-1,q)) \rightarrow \pi_1(SO^+(p,q)) \rightarrow 0
\]
by the induction hypothesis we then get
\begin{align*}
\boxed{\pi_1(SO^+(p,q))} &= \pi_1(SO^+(p-1,q))\\
&= \pi_1(SO(p-1)) \times \pi_1(SO(q)) = \boxed{\pi_1(SO(p)) \times \pi_1(SO(q)) }
\end{align*}
for $p > 3$. The last equality is due to $\pi_1(SO(q-1)) = \pi_1(SO(q))$ for $q>3$.
\end{itemize}
All that remains is the case $\pi_1(SO^+(3,3))$: \\\\
\begin{lem*}$SO^+(3,3) = SL(4, \mathbb{R}) / \{\pm 1\}$.
\end{lem*}
\begin{proof}
The fundamental representation of $SL(4,\bb R)$, $V$, induces a representation on $U = \bigwedge^2 V$ of dimension ${4 \choose 2} = 6$. $U$ has a symmetric bilinear form defined by
\[
(v_1 \wedge v_2, v_3 \wedge v_4) = (v_1 \wedge v_2 \wedge v_3 \wedge v_4)/(e_1 \wedge e_2 \wedge e_3 \wedge e_4)
\]
where $(v_1 \wedge v_2 \wedge v_3 \wedge v_4) \in \bigwedge^4 V\cong \bb R$. For all $g \in SL(4,\R)$,
\begin{align*}
1 = &\det g = g(v_1 \wedge v_2 \wedge v_3 \wedge v_4)/(v_1 \wedge v_2 \wedge v_3 \wedge v_4)\\
&\implies \quad g(v_1 \wedge v_2 \wedge v_3 \wedge v_4) = (v_1 \wedge v_2 \wedge v_3 \wedge v_4)
\end{align*}
so $\bigwedge^4 V$ is invariant under the action by $SL(4,\R)$, i.e. the above symmetric form $(\cdot,\cdot )$ is preserved under $SL(4,\R)$. Now we compute the signature of $(\cdot,\cdot )$ by looking at the canonical basis of $\bigwedge^2 V$,
\begin{align*}
\lvert e_1 \wedge e_2 \pm e_3 \wedge e_4 \rvert ^2 &= \pm 2(e_1 \wedge e_2, e_3 \wedge e_4) = \pm 2\\
\lvert e_1 \wedge e_3 \pm e_2 \wedge e_4 \rvert ^2 &= \pm 2(e_1 \wedge e_3, e_2 \wedge e_4) = \mp 2\\
\lvert e_1 \wedge e_4 \pm e_2 \wedge e_3 \rvert ^2 &= \pm 2(e_1 \wedge e_4, e_2 \wedge e_3) = \pm 2 
\end{align*}
so the signature is $(3,-3)$. This means $SO(3,3)$ is also exactly the linear maps that preserves the symmetric form $(\cdot,\cdot )$ of determinant 1. Since $\dim SO(3,3) = 15 = \dim SL(4, \bb R)$, the connected component $SO^+(3,3)$ must be isomorphic to the connected component\\
$SL(4, \mathbb{R}) / \{\pm 1\}$, as required.
\end{proof}
The above of course implies that $\pi_1(SO^+(3,3)) = \pi_1(SL(4,\bb R)/\{\pm 1\})$. It is known\footnote{To see this, note that $GL_n^+(\bb R)$ deformation retracts to both $SL_n(\bb R)$ and $SO(n)$, via Gram-Schmidt.} that $\pi_1(SL_4(\bb R)) \cong \bb Z/2 \bb Z$, which means  $|\pi_1(SO^+(3,3))| = |\pi_1(SL_4(\bb R)/\{\pm 1\})| = 4$, i.e. $\pi_1(SO^+(3,3)) = \bb Z/2 \bb Z \times \bb Z/2 \bb Z$ or $\bb Z/4 \bb Z$, we just need to determine which one.\\\\
Since $\pi_1(X^+_{3,q}) = 0, \pi_2(X^+_{3,q}) = \mathbb{Z}$, for $q = 0$ we get 
\[
\begin{tikzcd}
\cdots \arrow[r]& {\pi_2(SO^+(3,0))} = 0 \arrow[r] &  \pi_2(X^+_{3,0}) = \bb Z \ar[overlay, out=-10, in=170]{dll} \\
\pi_1(\underbrace{SO^+(2,0)}_{= SO(2)}) \arrow[r] & {\pi_1(\underbrace{SO^+(3,0)}_{= SO(3)})} \arrow[r] & \pi_1(X^+_{3,0}) = 0
\end{tikzcd}
\]
which simplifies to the following short exact sequence,
\begin{equation}
0 \rightarrow \mathbb{Z} \xrightarrow{g} \pi_1(SO(2)) \xrightarrow{f} \pi_1(SO(3)) \rightarrow 0
\end{equation}
equivalent to 
\[
0 \rightarrow \mathbb{Z} \xrightarrow{g} \bb Z \xrightarrow{f} \bb Z/ 2\bb Z \rightarrow 0.
\]
Since this sequence is exact, there is only one possibility for a surjective map $g: \bb Z \to \bb Z/ 2\bb Z$ which is 
\[
g(z) = z \mod 2
\]
This map is used below. Using a similar argument for $q = 3$ we can also get the short exact sequence
\begin{equation}
0 \rightarrow \mathbb{Z} \rightarrow \pi_1(SO^+(3,2)) \rightarrow \pi_1(SO^+(3,3)) \rightarrow 0
\end{equation}
which is also used below. 

Now there is a commutative diagram
\[\begin{tikzcd}
SO^+(3,2) \arrow[hookrightarrow]{r}  & SO^+(3,3) \\ 
SO(2) \times SO(3) \arrow[hookrightarrow]{r} \arrow[hookrightarrow]{u}{\rotatebox{90}{\(\simeq\)}} & SO(3) \times SO(3) \arrow[hookrightarrow]{u}
\end{tikzcd}\]
where every arrow is an inclusion, given explicitly by
\[
\begin{tikzcd}
\begin{pmatrix} B & 0\\0 & A \end{pmatrix} \arrow[r, maps to]      & \begin{pmatrix} 1 & 0& 0\\0 & A & 0\\0 &0 & B \end{pmatrix}                       \\
{(A,B)} \arrow[r, maps to] \arrow[u, maps to] \arrow[r, maps to] & {\left(\begin{pmatrix} 1 & 0\\0 & A\end{pmatrix},B\right)} \arrow[u, maps to].
\end{tikzcd}
\]
This, combined with $(3.6)$ then gives
\[\begin{tikzcd}[sep=1.8em, font=\small]
0 \arrow{r} & \mathbb{Z} \arrow{r} & \pi_1(SO^+(3,2)) \arrow{r}{\gamma}  & \pi_1(SO^+(3,3)) \arrow{r} & 0 \\ 
&& \pi_1(SO(2)) \times \pi_1(SO(3)) \arrow{r}{g} \arrow{u}{\rotatebox{90}{\(\simeq\)}} & \pi_1(SO(3)) \times \pi_1(SO(3)) \arrow{u} \arrow{r} & 0.
\end{tikzcd}\]
From $(3.5)$ we see that $g$ must necessarily be the map
\[
g : (z,x) \mapsto (z \mod 2,x)
\]
since the map in the first coordinate comes from $(3.5)$, in other words $g$ must be surjective. Therefore the image of $\gamma$ is some image of $\pi_1(SO(3)) \times \pi_1(SO(3)) = \mathbb{Z}/2\mathbb{Z}\times \mathbb{Z}/2\mathbb{Z}$. Finally, from the sequence
\[
0 \rightarrow \mathbb{Z} \xrightarrow{f} \pi_1(SO^+(3,2)) = \mathbb{Z}\times  \mathbb{Z}/2\mathbb{Z} \rightarrow \pi_1(SO^+(3,3)) = \text{im}(\gamma) \rightarrow 0
\]
there are three cases: 
\begin{itemize}
\item $f: 1 \mapsto (2,0) \Rightarrow \text{im}(\gamma) = \mathbb{Z}/2\mathbb{Z} \times \mathbb{Z}/2\mathbb{Z}$
\item $f: 1 \mapsto (1,1) \Rightarrow \text{im}(\gamma) = \mathbb{Z}/2\mathbb{Z}$
\item $f: 1 \mapsto (2,1) \Rightarrow \text{im}(\gamma) = \mathbb{Z}/4\mathbb{Z}$.
\end{itemize}
We discussed above that $\text{im}(\gamma) = \mathbb{Z}/2\mathbb{Z} \times \mathbb{Z}/2\mathbb{Z} \text{ or } \mathbb{Z}/4\mathbb{Z}$ so the second option is not possible. We also showed $\text{im}(\gamma)$ is some image of $\mathbb{Z}/2\mathbb{Z}\times \mathbb{Z}/2\mathbb{Z}$ under a surjective map, so the third option is also not possible. Thus we conclude
\[
\pi_1(SO^+(3,3)) = \mathbb{Z}/2\mathbb{Z} \times \mathbb{Z}/2\mathbb{Z} = \pi_1(SO(3)) \times \pi_1(SO(3))
\]
as required, thus completing the final case of the proof.
\end{proof}
\newpage
\section{The fundamental group of $SO^+(p,q)$ via deformation retract}
For the reader's convenience, we reproduce two standard routes to calculating the above formula. 

The first method appeals to (an abbreviated version of, see \cite{knapp} Chapter 6 for general treatment) the Cartan decomposition. It is easy to see that 
\[
\mathfrak{so}(p,q) = \{X \in Mat_n(\bb R)\ |\ X^TI_{p,q} = -I_{p,q}X\}.
\]
Define the map $\theta: \mathfrak{so}(p,q) \to \mathfrak{so}(p,q)$ sending $X \mapsto -X^T$ (Also observe that because $Lie(G)$ is a local construction, $Lie(G^+) \cong Lie(G)$). $\theta$ is an involution and the bilinear form $-B(-,\theta(-))$, where $B$ is the Killing form, is positive definite, so $\mathfrak{so}(p,q) = \mathfrak k \oplus \mathfrak p$, where $\mathfrak k$ is the positive 1 eigenspace and $\mathfrak p$ is the negative 1 eigenspace of $\theta$. The bracket obeys the relations
\[
[\mathfrak k, \mathfrak k] \subset \mathfrak k,\quad [\mathfrak k, \mathfrak p ]\subset \mathfrak p,\quad [\mathfrak p, \mathfrak p]\subset \mathfrak k
\]
In particular, because the $\theta$-eigenvalue of a bracket is the product of the $\theta$-eigenvalues. For this reason, the above decomposition is only in terms of vector spaces, not Lie subalgebras: $\mathfrak p$ isn't even a Lie subalgebra. Such a decomposition obeying the commutation relations arising from a map $\theta$ (\underline{Cartan involution}) is called a \underline{Cartan decomposition}.

We would like to lift this to a decomposition of the Lie group, $SO^+(p,q)$, by using the exponential. The difficult work in this approach is to analyze the exponential on both factors. In this case, one is quite tractable, while the other must be handled with the powerful general theory.
\begin{lem}
\[
\mathfrak k \cong \mathfrak{so}(p) \times \mathfrak{so}(q)
\]
\end{lem}

\begin{proof}From the expression given for $\mathfrak{so}(p,q)$, any matrix has block decomposition $X = \begin{pmatrix} A & B \\ C & D \end{pmatrix}$ where $B = C^T$ and $A \in \mathfrak{so}(p)$, $D \in \mathfrak{so}(q)$. To lie in the $+1$ eigenspace, it must be that $X =-X^T \Rightarrow B = C = 0$, from which the result follows. 
\end{proof}
It follows that the image $\exp(\mathfrak k) \subset SO(p) \times SO(q)$.
\begin{lem}\textit{$SO(n)$ is a compact subspace of $\bb R^{n^2}$ for all $n$.}
\end{lem}

\begin{proof}By the Heine-Borel theorem, it suffices to show that $SO(n)$ is closed and bounded. For closedness, observe that $SO(n)$ is the zero set of a finite number of polynomial equations in the coordinates (matrix entries), namely $AA^T=A^TA=Id$ and $det\ A = 1$, since matrix multiplication and determinant operations are polynomial. Further, the columns of any matrix belonging to $SO(n)$ must form an orthonormal basis of $\bb R^n$, so $SO(n)$ is bounded.
\end{proof}
Since $SO(p),SO(q)$ are compact, $\exp$ is surjective, so $\exp(\mathfrak k) = SO(p) \times SO(q) \equiv K$. Then the Cartan decomposition on the level of Lie groups takes the form of a diffeomorphism
\[
K\times \mathfrak p \overset\sim\longrightarrow G
\]
via the map $(k,p) \mapsto k \cdot \exp(p)$, so that every matrix $A \in SO^+(p,q)$ can be written uniquely in the form 
\[
A = Q\exp(P)
\]
for $Q \in K$ and $P \in \mathfrak p$. In general to prove the map above is a diffeomorphism is very involved, and a full proof (along with other statements in the same context) can be found in \cite{knapp}. This is a version of polar decomposition for the indefinite orthogonal group. In particular, the quotient $G/K$, which is generally a Riemannian symmetric space, is affine, so the projection $G\to G/K$ is a homotopy equivalence, thus inducing isomorphisms
\[
\pi_n(SO^+(p,q)) \cong \pi_n(SO(p)) \times \pi_n(SO(q)) \quad \quad n \geq 0.
\]
One advantage of the approaches taken in this section is that a deformation retract preserves \textit{all} homotopy groups, not just the fundamental group, although there is no general formula for the right hand side.
\\\\
Another decomposition very related to the Cartan decomposition is the Iwasawa decomposition. In this case, reference to the decomposition already brings one close to the desired result.
\begin{thm}[Iwasawa-Malcev,\cite{stroppel}] Let $G$ be a locally compact group such that $G/G^+$ is compact. Then the followings hold.
\begin{enumerate}[label=(\roman*)]
    \item Every compact subgroup of $G$ is contained in some maximal compact subgroup of $G$.
    \item The maximal compact subgroups of $G$ form a single conjugacy class.
    \item There exists some natural number $n$ such that the underlying topological space of $G$ is homeomorphic to $\bb R^n \times K$, where $K$ is one of the maximal compact subgroups of $G$.
    \item In particular, every maximal compact subgroup of a locally compact connected group is connected.
\end{enumerate}
\end{thm}
In particular, $(ii), (iii)$ imply that $SO^+(p,q)$ has the same homotopy type as any of its maximal compact subgroups. Thus it only remains to identify one such.
\begin{lem} $SO(p) \times SO(q)$ is a maximal compact subgroup of $SO^+(p,q)$.
\end{lem}
\begin{proof} First recall $SO(p)$ and $SO(q)$ are compact, so $SO(p) \times SO(q)\subset SO^+(p,q)$ is compact. To see that each such is maximal amongst compact subgroups, we follow \cite{conrad}.
\end{proof}

First a warm-up, whose result we will use later: 
\begin{lem} $SO(n)$ is a maximal compact subgroup of $SL(n)$.
\end{lem}
\begin{proof} We first show the ``larger'' result, that $O(n)$ is a maximal compact subgroup of $GL(n)$. We take this approach to emphasize that the ``speciality'' of $SO(n)$ and $SL(n)$ is mostly irrelevant to maximal compactness. The properties that are essential here are those of orthogonality and invertibility.

Let $K$ be a compact subgroup of $GL(n)$ containing $O(n)$. Because $K$ is compact, we can consider the (normalized) Haar measure on $K$ (the existence of such is a result of basic Lie theory, proven for example in \cite{lee}), $dk$, and define a positive definite inner product on $\bb R^n$ by first choosing an arbitrary positive definite inner product $\langle-,-\rangle$, and defining the desired inner product by the formula
\[
\langle x,y\rangle_K :=\int_K\langle kx,ky\rangle dk
\]
which is manifestly $K$-invariant (the finiteness of this integral/existence of Haar measure is why $O(n)$ is only maximal amongst \textit{compact} subgroups). Therefore $K$ preserves this new inner product, so that $K$ is a subgroup of some special orthogonal group corresponding to this new inner product. Then $K$ must be conjugate to a subgroup of $O(n)$, since the set of all orthogonal groups indexed by positive-definite inner products form a single conjugacy class. This is a standard result in linear algebra, but the spirit of the proof is that all positive definite inner products can be brought into the standard form by change of coordinates, for example applying Gram-Schmidt and Sylvester's law of inertia. Therefore we have inclusions 
\[
g\cdot O(n) \cdot \inv g \subset g \cdot K \cdot \inv g \subset O(n)
\] 
but $g \cdot O(n) \cdot \inv g = O(n)$, since for example they have the same dimension and number of connected components. Thus $g\cdot K \cdot \inv g = O(n)\Rightarrow K = g \cdot O(n) \cdot \inv g = O(n)$. We have shown $K \supset O(n) \Rightarrow K = O(n)$, so $O(n)$ is maximal among compact subgroups.

If one runs this proof again for $SL(n)$ mutatis mutandis, one shows that $SO(n)$ is a maximal compact subgroup. 
\end{proof}
Now we return to the original lemma. 
\begin{lem}
$SO(p) \times SO(q)$ is a maximal compact subgroup of $SO^+(p,q)$.
\end{lem}
\begin{proof}Again we begin by proving $O(p) \times O(q)$ is a maximal compact subgroup of $O(p,q)$, defined in the expected way. Suppose we have $K$ compact with $O(p) \times O(q) \subset K\subset O(p,q)$. In particular, the action of $K$ on $\bb R^n$ preserves the indefinite inner product (the standard one of signature $(p,q)$.), which we denote as $Q$. $Q$ has an associated non-degenerate symmetric bilinear form given by 
\[
\langle x,y\rangle = \sum_{i=1}^p x_iy_i - \sum_{j=p+1}^{p+q}x_jy_j
\]
inducing an isomorphism $\bb R^n \to (\bb R^n)^*$ by the map $v \mapsto \langle v,-\rangle$.

Since $K$ is a compact subgroup of $GL(p+q)$, following from our proof of the lemma above (the definite case), $K$ is contained in $O(Q')$ for $Q'$ some $K$-invariant positive definite quadratic form with associated $K$-invariant positive definite inner product $\langle-,-\rangle_K$\footnote{Note that at this point in the above proposition, we were basically done, since we only needed to produce a positive definite $K$-invariant inner product. Here we seek an indefinite inner product which cannot possibly arise directly from this construction, so there is more work to do.}. Thus there is another isomorphism $v \mapsto \langle v, -\rangle_K$, and thus a composite isomorphism $f: \bb R^n \to \bb R^n$, where $f(v)$ is the unique vector such that $\langle v, -\rangle = \langle f(v),-\rangle_K$. 
\end{proof}
\begin{lem}
$f$ is self-adjoint with respect to $\langle -,-\rangle_K$.
\end{lem}
\begin{proof}By construction ($f$ is the map such that this holds)
\[
\langle f(v),w\rangle_K = \langle v,w\rangle
\]
Similarly 
\[
\langle v,f(w)\rangle_K = \langle f(w),v\rangle_K = \langle w,v\rangle = \langle v,w\rangle
\]
as required.
\end{proof}
Thus $f$ has all real and strictly positive eigenvalues, and by the spectral theorem for real matrices, $f$ is diagonalizable. $f$ commutes with $K$, $kf(v) = f(kv)$: This is because $f(kv)$ is the unique element of $\R^{n}$ satisfying
\[
\langle kv,-\rangle = \langle f(kv),-\rangle_K
\]
but
\[
\langle kv,-\rangle = k\langle v,-\rangle = k\langle f(v),-\rangle_K = \langle kf(v),-\rangle_K
\]
so by uniqueness
\[
f(kv) = kf(v)
\]
so that for any $k \in K$, $f$ and $k$ are commuting linear maps with $f$ diagonalizable, so that $K$ preserves each $f$-eigenspace. If $v\ne 0 \in V_{\lambda}^{f}$, 
\[
\lvert v \rvert ^{2} \equiv \langle v,v\rangle = \langle f(v),v\rangle_{K} = \langle \lambda v,v\rangle_{K} \equiv \lambda \lvert v \rvert_{K}^{2}
\]
which shows that $q$ is definite when restricted to $V_{\lambda}^{f}$, with sign equal to that of $\lambda$ (recall $\lambda$ is real and non-zero). Then there is a decomposition
\[
\bb R^n \cong \left(\bigoplus_{\lambda > 0} V_{\lambda}^{f} \right) \bigoplus \left(\bigoplus_{\mu<0} V_{\mu}^{f}\right)
\]
as positive eigenvalue eigenspaces and negative eigenvalue eigenspaces, with $Q$ positive definite on the first and negative definite on the second. There can be no 0-eigenspace since $f$ is an isomorphism. Since $K$ preserves each eigenspace, $K$ preserves the first and second components, so that $\bb R^{n}$ is a direct sum of a space on which $Q$ is positive definite and a space on which $Q$ is negative definite. Again by Sylvester's theorem, that implies the dimension of the first component is $p$ and the dimension of the second component is $q$. Thus $K$ fixes an indefinite form of signature $(p,q)$. Similar to the argument above, the indefinite orthogonal groups indexed by forms of signature $(p,q)$ all belong to one conjugacy class\footnote{To do this carefully, one performs Gram-Schmidt on both portions simultaneously in a ``compatible'' way. The details are straightforward and left to the reader, or can be checked in \cite{conrad}.}, so that $K$ is conjugate to a subgroup of $O(p) \times O(q)$, and is thus equal to $O(p) \times O(q)$, so $O(p) \times O(q)$ is maximal compact in $O(p,q)$.\\\\
\begin{proof}[Proof of main result] To promote $O(p,q)$ to $SO(p,q)$, we note that the candidate subgroups $O(p) \times O(q)$ do not belong to $SO(p) \times SO(q)$. The underlying embedding is the map $(A,A') \in O(p) \times O(q) \mapsto \text{diag}(A,A')$, which clearly always belongs to $O(p,q)$, but not necessarily $SO(p,q)$ unless $\det A = \det A'$. Thus our candidate subgroup is 
\[
K_{p,q} :=\{(A,A') \in O(p) \times O(q) \mid \det A = \det A' \}
\]
and the previous arguments all go through mutatis mutandis to show that $K_{p,q}$ is a maximal compact subgroup of $SO(p,q)$. Then $K_{p,q}^{+}$ is a maximal compact subgroup of $SO^{+}(p,q)$. It is easy to see that $K_{p,q}^{+} = SO(p) \times SO(q)$, by continuity of the determinant and the observation $(Id_{p},\ Id_{q}) \in K_{p,q}^{+}$.
\end{proof}

\end{document}